\UseRawInputEncoding
\documentclass[12pt, twoside]{article}

\usepackage{amsmath,amsthm,amssymb,color}
\usepackage{times}
\usepackage{enumerate}
\usepackage{empheq}
\usepackage{url}
\usepackage{lineno}
\usepackage{microtype}
\usepackage{cases}
\usepackage[american]{babel}
\usepackage{indentfirst}
\def\titlerunning#1{\gdef\titrun{#1}}
\makeatletter
\def\author#1{\gdef\autrun{\def\and{\unskip, }#1}\gdef\@author{#1}}
\def\address#1{{\def\and{\\\hspace*{18pt}}\renewcommand{\thefootnote}{}%
\footnote {#1}}%
\markboth{\autrun}{\titrun}}
\makeatother
\def\email#1{e-mail: #1}
\def\subjclass#1{{\renewcommand{\thefootnote}{}%
\footnote{\emph{Mathematics Subject Classification (2020):} #1}}}

\newtheorem{thm}{Theorem}
\newtheorem{prop}{Proposition}
\newtheorem{lem}{Lemma}
\newtheorem{de}{Definition}
\newtheorem{re}{Remark}
\newtheorem{ex}{Example}
\numberwithin{equation}{section}

\DeclareMathOperator*{\supp}{supp}

\newcommand{\g}{\gamma}
\newcommand{\R}{\mathbb{R}}
\newcommand{\T}{\mathbb{T}}
\newcommand{\s}{\tau}

\newcommand{\p}{\mathcal{P}(\mathbb{T}^d)}

\begin{document}
\baselineskip=16pt

\titlerunning{First Order Discounted Mean Field Games}
\title{A semi-discrete approximation for first-order stationary mean field games}
\author{Renato Iturriaga \and Kaizhi Wang}

\maketitle

\address{Renato Iturriaga: Centro de Investigaci\'on en Matem\'aticas, Guanajuato, Mexico;
\email{renato@cimat.mx}
\and Kaizhi Wang: School of Mathematical Sciences, Shanghai Jiao Tong University, Shanghai 200240, China;
\email{kzwang@sjtu.edu.cn}
}

\begin{abstract}
We provide an approximation scheme for first-order stationary mean field games with a separable Hamiltonian. First, we discretize Hamilton-Jacobi equations by discretizing in time, and then prove the existence of minimizing holonomic measures for mean field games. At last, we obtain two sequences of solutions $\{u_i\}$ of discrete Hamilton-Jacobi equations and minimizing holonomic measures $\{m_i\}$ for mean field games and show that $(u_i,m_i)$ converges to a solution  of the stationary mean field games.

\end{abstract}

\tableofcontents


\section{Introduction}
\setcounter{equation}{0}
Mean field games \cite{H,LL1,LL2,LL} consists of studying the global behavior of systems composed of infinitely many agents which interact in a symmetric manner. A first-order mean field games model is a coupled system of partial differential equations, one Hamilton-Jacobi equation and one continuity equation. Here, we focus on the first-order ergodic (or stationary) mean field games system
\begin{subequations}\label{lab1}
	\begin{empheq}[left=\empheqlbrace]{align}
		&  H(x, Du)=F(x, m)+c(m) \quad \text{in} \quad \mathbb{T}^d, \\[1mm]
		& \text{div}\Big(m \frac{\partial H}{\partial p}(x, Du)\Big)=0    \quad\quad\quad \text{in} \quad \mathbb{T}^d, \\[1mm]  & \int_{\mathbb{T}^d}{m\ dx}=1.
	\end{empheq}
\end{subequations}
This system arises in the study of the long-time behavior problem of first-order mean field games with finite horizon \cite{C3}.  In this work, we aim to study a semi-discrete in time approximation of the first-order ergodic  mean field games system \eqref{lab1}. We are concerned with the convergence of the discrete scheme. In a forthcoming paper we will deal with a fully discrete approximation problem for \eqref{lab1}, where space discretization will be added. See \cite{CS1}, \cite{CS2} for semi-discrete and fully discrete approximation schemes for first-order evolutionary mean field games with finite horizon,  respectively.  We refer the readers to \cite{A0,A1,A2,A3,A4,A5,B} and the references therein for numerical methods and convergence results of different discrete schemes for second-order mean field games.

\subsection{Assumptions and main results}
Let $H(x,p): \mathbb{T}^d\times\mathbb{R}^d\to\mathbb{R}$ be a $C^2$ Hamiltonian satisfying  Tonelli conditions, where $\T^d:=\R^d/\mathbb{Z}^d$ denotes the standard flat torus. The associated Lagrangian is defined by
\[
L(x,v):=\sup_{v\in\mathbb{R}^d}\big(\langle p, v\rangle-H(x,p)\big),\quad (x,v)\in\T^d\times\R^d.
\]
Then $L$ satisfies:
\begin{itemize}
	\item [\textbf{(L1)}] \textbf{Strict convexity}: {\it for each $(x, v) \in \T^d\times\R^d, \frac{\partial^{2} L}{\partial v^{2}}(x, v)$ is positive definite;
	}
\item [\textbf{(L2)}] \textbf{Superlinearity}: for each $K>0$, there is $C(K)\in\R$ such that
\[
L(x,v)\geqslant K|v|+C(K),\quad \forall (x,v)\in \T^d\times\R^d.
\]
\end{itemize}

Let $\p$ denote the set of probability measures on $\T^d$.
Let the coupling term $F:\T^d\times \mathcal{P}(\mathbb{T}^d) \to \mathbb{R}$ be a function, satisfying the following  assumptions:
\begin{itemize}
	\item[\textbf{(F1)}] for each $m \in \mathcal{P}(\mathbb{T}^d)$, the function $x \mapsto F(x,m)$ is of class $C^{2}$, and there is a constant $F_\infty>0$ such that 	
	\begin{align*}
	\|F(\cdot,m)\|_\infty,\, \|D_xF(\cdot,m)\|_\infty\leqslant F_\infty,\quad \forall m\in\p,
	\end{align*}
	where  $\|\cdot\|_\infty$ denotes the supremum norm;
	\item[\textbf{(F2)}] $F(\cdot,\cdot)$ and $D_xF(\cdot,\cdot)$ are continuous on $\mathbb{T}^d\times \mathcal{P}(\mathbb{T}^d)$;
	\item[\textbf{(F3)}] there is a constant $\mathrm{Lip}(F)>0$ such that
	\[
	|F(x,m_1)-F(x,m_2)|\leqslant \mathrm{Lip}(F){d_{1}(m_1, m_2)},\quad \forall x\in\T^d,\, \forall m_1,m_2\in\p,
	\]
	where the distance $d_1$ is the Kantorovich-Rubinstein distance.
\end{itemize}

\begin{ex}
	Let $F(x,m)=f(x)g(m)$, where $f:\T^d\to\R$ is of class $C^2$, and $g:\p\to\R$ is Lipschitz. Then 
	\[
	\|F(\cdot,m)\|_\infty=\|f(\cdot)g(m)\|_\infty,\, \|D_xF(\cdot,m)\|_\infty=\|Df(\cdot)g(m)\|_\infty\leqslant F_\infty,\quad \forall m\in\p,
\]
for some $F_\infty>0$; $F(\cdot,\cdot)=f(\cdot)g(\cdot)$ and $D_xF(\cdot,\cdot)=Df(\cdot)g(\cdot)$ are continuous on $\T^d\times\p$; for each $x\in\T^d$, each  $m_1$, $m_2\in\p$,
\[
	|F(x,m_1)-F(x,m_2)|=|f(x)(g(m_1)-g(m_2))|\leqslant\|f\|_\infty\mathrm{Lip}(g)d_{1}(m_1, m_2).	
	\]	
\end{ex}

\begin{de}\label{ers}
	A solution of the mean field games system \eqref{lab1} is a couple $(u, m) \in C(\mathbb{T}^d) \times \mathcal{P}(\mathbb{T}^d)$ such that (1.1a) is satisfied in viscosity sense and (1.1b) is satisfied in distributions sense.
\end{de}
	
	\begin{re}
		Let us recall the definition of viscosity solutions of (1.1a) and the one of solutions of (1.1b) in distributions sense here.

	 A function $u:\T^d\rightarrow \mathbb{R}$ is called a viscosity subsolution of equation (1.1a), if for every $C^1$ function $\varphi:\T^d\rightarrow\mathbb{R}$ and every point $x_0\in \T^d$ such that $u-\varphi$ has a local maximum at $x_0$, we have
		\[
		H(x_0,D\varphi(x_0))\leqslant F(x_0,m)+c(m);
		\]
		A function $u:\T^d\rightarrow \mathbb{R}$ is called a viscosity supersolution of equation (1.1a), if for every $C^1$ function $\psi:\T^d\rightarrow\mathbb{R}$ and every point $y_0\in \T^d$ such that $u-\psi$ has a local minimum at $y_0$, we have
		\[
		H(y_0,D\psi(y_0))\geq F(y_0,m)+c(m);
		\]
	A function $u:\T^d\rightarrow\mathbb{R}$ is called a viscosity solution of equation (1.1a) if it is both a viscosity subsolution and a viscosity supersolution.

We say that a measure $m\in\p$ satisfies (1.1b) in the sense of distributions, if	
$$
	\int_{\T^d}\left\langle D f(x), \frac{\partial H}{\partial p}(x, D u(x))\right\rangle \mathrm{d} m(x)=0, \quad \forall f \in C^{\infty}\left(\T^d\right).
	$$
	\end{re}

For each $m\in\p$, $H_m(x,p):=H(x,p)-F(x,m)$ is a Tonelli Hamiltonian defined on $\T^d\times\R^d$. Denote by $L_m$ the associated Lagrangian, i.e.,
\[
L_m(x,v)=L(x,v)+F(x,m),\quad (x,v)\in\T^d\times\R^d.
\]
Denote by $\Phi^{L_m}_t$ and $\Phi^{H_m}_t$ the Euler-Lagrange flow of $L_m$ and the Hamiltonian flow of $H_m$, respectively.

\begin{re} Assume (L1), (L2) and (F1).
	\begin{itemize}
		\item [(i)] We used $c(m)$ in \eqref{lab1} to denote the Ma\~n\'e critical value \cite{Ma} of $H_m$. It is well known that for any given $m\in\p$, $c(m)$ is the unique real number $k$ such that equation $H_m(x,Du(x))=k$ has viscosity solutions.
		\item [(ii)] Let us recall the notion of Mather measures for Tonelli Lagrangians introduced by Mather in \cite{Mat91}. A measure $\mu\in\mathcal{P}(\T^d\times\R^d)$ is called a Mather measure for $L_{m}$, if it satisfies
	\[
	\int_{\T^d\times\R^d}L_{m}(x,v)d\mu=\min_{\mu}\int_{\T^d\times\R^d}L_{m}(x,v)d\mu=-c(m),
	\]
	where the minimum is taken over the set of all Borel probability measures on $\T^d\times\R^d$ invariant under the Euler-Lagrange flow $\Phi^{L_{m}}_{t}$. Let $\mathcal{P}^{\ell}(\T^d\times\R^d)$ be the set of probability measures on the Borel $\sigma$-algebra of $\T^d\times\R^d$ such that $\int_{\T^d\times\R^d}|v| d \mu<+\infty$. Define the set of closed measures on $\T^d\times\R^d$ as
	\[
	\mathcal{K}(\T^d\times\R^d):=\left\{\mu \in \mathcal{P}^{\ell}(\T^d\times\R^d): \int_{\T^d\times\R^d} vD \varphi(x) d \mu=0,\, \forall \varphi \in C^{1}(\T^d)\right\}.
	\]
	A closed measure $\mu$ satisfying  $\int_{\T^d\times\R^d}L_{m}(x,v)d\mu=-c(m)$ is a Mather measure.
	\item [(iii)]  Let $m_0\in\p$ be such that there is a Mather measure $\mu_0$ for $L_{m_0}$ with $m_0=\pi\sharp \mu_0$,  where $\pi: \T^d \times \mathbb{R}^{d} \rightarrow \T^d$ denotes the canonical projection, and $\pi\sharp \mu_0$ denotes the push-forward of $\mu_0$ through $\pi$. Let $u_0$ be any viscosity solution of $H(x,Du)=F(x,m_0)+c(m_0)$. Then $m_0$ satisfies $\text{div}\big(m \frac{\partial H}{\partial p}(x, Du_0)\big)=0$ in distributions sense. See \cite{HW} for details.
	\end{itemize}

\end{re}

\begin{re}
		A $\mathbb{Z}^d$-periodic function $u\in C(\R^d)$ is a viscosity solution of (1.1a) if and only if
		\begin{align}\label{5-2}
		u(y)-c(m)t=\inf_{x\in\R^d}\Big(u(x)+h^m_t(x,y)\Big),\quad \forall y\in\R^d,\, \forall t>0,
		\end{align}
		where
		\begin{align}\label{6-1}
		h_t^m(x,y):=\inf_{\g}\int_0^t L_m(\g,\dot{\g})ds,
		\end{align}
		where the infimum is taken among the continuous and piecewise $C^{1}$ paths $\gamma:[0, t] \rightarrow \R^d$ with $\g(0)=x$, $\g(t)=y$. See, for instance, \cite{Fat-b} for a proof. We call $h^m_t(x,y)$ the  minimal action function and the curves achieving the infimum in \eqref{6-1} minimizing curves of $L_m$ with the action $h_t^m(x,y)$.
\end{re}

For each $\s>0$ and each $m\in\p$, define the discrete action function by
\[
	\mathcal{L}_{\tau,m}(x,y):=\tau \big(L(x,\frac{y-x}{\tau})+F(x,m)\big),\quad \forall x, y\in\R^d.
	\]
According to \cite[Theorem 4.3]{GT} and \cite[Theorem 9]{ST}, under assumptions (L1), (L2) and (F1) one can deduce that for each $\s>0$, each $m\in\p$, there is a unique constant $\bar{L}(\s,m)\in\R$,  such that the discrete Lax-Oleinik equation
\begin{align}\label{de}
u_{\s,m}(y)+\s\bar{L}(\s,m)=\inf_{x\in\R^d}\big(u_{\s,m}(x)+\mathcal{L}_{\s,m}(x,y)\big), \quad  \forall y\in\R^d,
\end{align}
has continuous $\mathbb{Z}^d$-periodic solutions $u_{\s,m}$, and $\bar{L}(\s,m)\to-c(m)$ as $\s\to0$. The authors of  \cite{ST} showed the convergence of a subsequence of solutions of  discrete Lax-Oleinik equations.

Ma\~n\'e \cite{Ma1} introduced the notion of holonomic measures in his study of Mather theory. It has great advantage of dealing with different Lagrangians at the same time. Here we will use a discrete version of the notion of holonomic measures.
\begin{de}
	We say that a probability measure $\mu \in \mathcal{P}\left(\mathbb{T}^{d} \times \mathbb{R}^{d}\right)$ is $\s$-holonomic, provided
	$$
	\int_{\mathbb{T}^{d} \times \mathbb{R}^{d}} \varphi(x+\tau v) \mathrm{d} \mu(x, v)=\int_{\mathbb{T}^{d} \times \mathbb{R}^{d}} \varphi(x) \mathrm{d} \mu(x, v)
	$$
	for any $\varphi\in C(\mathbb{T}^{d})$. The set of $\s$-holonomic measures is denoted by $\mathcal{P}_{\tau}\left(\mathbb{T}^{d} \times \mathbb{R}^{d}\right)$.
\end{de}

In view of \cite[Definition 3.5 and Theorem 4.3]{GT}, we know that for each $\s>0$, each $m\in\p$,
	\begin{align}\label{5-1}
	\bar{L}(\tau,m)=\min _{\mu} \int_{\mathbb{T}^{d} \times \mathbb{R}^{d}} L_m(x,v) d\mu,
	\end{align}
	where the minimum is taken over $\mathcal{P}_{\tau}\left(\mathbb{T}^{d} \times \mathbb{R}^{d}\right)$. A measure $\mu$ attaining the minimum  is called a minimizing $\s$-holonomic measure for $L_m$.

The main result of the present paper is stated as follows.
\begin{thm}\label{ma}
	Assume (L1), (L2) and (F1)-(F3). Then
	\begin{itemize}
  \item [(i)] For each $\s>0$, there is $m\in\p$ such that there exists a minimizing $\s$-holonomic measure $\eta_{\s,m}$ for the Lagrangian $L_m$ with
  \[
  m=\pi\sharp \eta_{\s,m}.
  \]
  Such a measure $m$ is denoted by $m_\s$ (maybe not unique).
  \item [(ii)] There is a subsequence $\s_i\to0$, a subsequence $m_{\s_i}\stackrel{w^*}{\longrightarrow} m_0$, and a subsequence $u_{\s_i,m_{\s_i}}$ solutions of \eqref{de}  such that $u_{\s_i,m_{\s_i}}$ converges to $u_0$ uniformly on $\T^d$ and $(u_0,m_0)$ is a solution of \eqref{lab1}.
\end{itemize}
\end{thm}

%

\begin{re} Outline of the proof of Theorem \ref{ma}:
\begin{itemize}
	\item [(i)] First, we discretize the continuous Lax-Oleinik equation \eqref{5-2} by discretizing in time. Analyzing the properties of solutions to the discrete Lax-Oleinik equation \eqref{de} is our starting point. Our discrete scheme is the mean field games analogue of the approximation scheme for Hamilton-Jacobi equations $H(x,Du)=c(H)$ considered in \cite{GT,ST}, where $c(H)$ is the Ma\~n\'e critical value of $H$.
	\item [(ii)] Next, we study the tightness of minimizing $\s$-holonomic measures for $L_m$ and
	 introduce the notion of minimizing $\s$-holonomic measures for mean field games. Based on the tightness result we get the existence of minimizing $\s$-holonomic measures for mean field games by using Kakutani fixed point theorem.
	\item [(iii)] At last, we get a convergent subsequence $(u_{\s_i,m_{\s_i}},m_{\s_i})$ whose limit is a solution of \eqref{lab1}. Theorem \ref{ma} can be regarded as a selection type result for \eqref{lab1}.
\end{itemize}
\end{re}

\subsection{Notations and definitions}
Now we introduce the symbols used in this paper.
Denote by $\mathbb{N}$ the set of positive integers, by $\mathbb{R}^{d}$ the $d$-dimensional real Euclidean space, by $\langle p, v\rangle$ or $pv$ the Euclidean scalar product of $p$ and $v$, by $|\cdot|$ the usual norm in $\mathbb{R}^{d}$, and by $B_{R}$ the open ball with center 0 and radius $R$. Let $\Omega \subset \T^d\times\mathbb{R}^{d}$. $\mathrm{cl}({\Omega})$ stands for its closure. We identify the tangent bundle $T\T^d$ and the cotangent bundle $T^*\T^d$ with $\T^d\times\R^d$.
 $C^{k}\left(\T^{d}\right)(k \in \mathbb{N})$ stands for the function space of $k$-times continuously differentiable functions on $\mathbb{T}^{d}$, and $C^{\infty}\left(\mathbb{T}^{d}\right):=\bigcap_{k=0}^{\infty} C^{k}\left(\mathbb{T}^{d}\right)$.
The spatial gradient of $F$ is denoted by $D_x F=\frac{\partial F}{\partial x}=\left(D_{x_{1}} F, \ldots, D_{x_{n}} F\right)$, where $D_{x_{i}} F=\frac{\partial F}{\partial x_{i}}, i=1,2, \ldots, d$. Given a metric space $(X, d)$ we denote by $\mathcal{B}(X)$ the Borel $\sigma$-algebra on $X$ and by $\mathcal{P}(X)$ the set of Borel probability measures on $(X, \mathcal{B}(X))$. The support of a measure $\mu \in \mathcal{P}(X)$, denoted by  $\operatorname{supp}(\mu)$, is the closed set defined by
$$
\operatorname{supp}(\mu):=\left\{x \in X: \mu\left(V_{x}\right)>0 \text { for each open neighborhood } V_{x} \text { of } x\right\}.
$$

 Let $X$ be a Polish space (complete, separable metric spaces, equipped with their Borel σ-algebra) endowed with a distance $d$. As mentioned above, we denote by $\mathcal{P}(X)$ the space of Borel probability measures.  $\mu_{k}\in \mathcal{P}(X)$ converges weakly to $\mu$ if for all $\varphi \in C_{b}(X)$ (i.e., $\varphi$ is bounded and continuous), $\int_X \varphi d \mu_{k}$ converges to $\int_X \varphi d \mu$ as $k \rightarrow +\infty$. This defines a separable, Hausdorff topology on $P(X)$, called the weak topology.

	Prokhorov theorem (see, for instance, \cite{V-b}) ensures that a subset $S$ of $\mathcal{P}(X)$ is relatively weakly compact if and only if it is tight, i.e. for all $\varepsilon>0$ there is a compact subset $K_{\varepsilon}$ of $X$ such that for all $\mu \in S, \mu(X\backslash K_{\varepsilon})\leq \varepsilon$.

	If $X$ is locally compact, then Riesz theorem identifies the space $M(X)$ of measures, normed by total variation, with the dual of the space $C_{0}(X)$ of continuous functions going to 0 at infinity. Then one can introduce the ``weak-* topology" on $P(X)$. At the level of probability measures, weak and weak-* convergences are  equivalent.

	Let $p \geq 0$ be a nonnegative real number. Denote by $\mathcal{P}_{p}(X)$ the set of probability measures with finite moments of order $p$, i.e. those measures $\mu$ such that for some (and thus any) $x_{0} \in X$,
	$$
	\int d\left(x_{0}, x\right)^{p} d \mu(x)<+\infty.
	$$
	If $d$ is bounded, then $\mathcal{P}_{p}(X)$ coincides with  $\mathcal{P}(X)$. Given $\mu$, $\nu\in \mathcal{P}_{p}(X)$, those probability measures $\pi\in\mathcal{P}(X\times X)$ that satisfy  
	\begin{align}\label{6-5}
		\pi(A \times X)=\mu(A), \quad \pi(X \times A)=\nu(A)
	\end{align}
	for all measurable subsets $A$ of $X$, are said to have marginals $\mu$ and $\nu$. Let $\Pi(\mu, \nu):=\{\pi \in P(X \times X): \eqref{6-5}\ \text{holds for all measurable}\ A\}$.
	Define the Monge-Kantorovich distance of order $p$ between $\mu$ and $\nu$  by
	\[
	d_{p}(\mu, \nu)=\left(\inf _{\pi \in \Pi(\mu, \nu)} \int_{X\times X} d(x, y)^{p} d \pi(x, y)\right)^{1 / p}.
	\]
	The Monge-Kantorovich distance of order $1$ will be also  called the Kantorovich-Rubinstein distance.

	Let us recall a very useful fact (see, for example, \cite{V-b}): let $p \in(0, +\infty)$, let $\left(\mu_{k}\right)_{k \in \mathbb{N}}$ be a sequence of probability measures in $\mathcal{P}_{p}(X)$, and let $\mu \in \mathcal{P}(X)$. Then, the following two statements are equivalent:
	(i) $d_{p}\left(\mu_{k}, \mu\right) \to 0$, as $k\to+\infty$
	(ii) $\mu_{k}$ converges weakly to $\mu$ as $k\to+\infty$, and $\mu_{k}$ satisfies the  tightness condition: for some (and thus any) $x_{0} \in X$,
	$$
	\lim _{R \rightarrow +\infty} \limsup _{k \rightarrow +\infty} \int_{d\left(x_{0}, x\right) \geq R} d\left(x_{0}, x\right)^{p} d \mu_{k}(x)=0.
	$$

The rest of the paper is organized as follows.  We provide some preliminary results in Section 2. Section 3
is devoted to the existence of minimizing holonomic measures for mean field games.
We show the convergence of the approximation scheme and that  the limit functions are solutions of \eqref{lab1} in Sections 4 and 5.

\section{A priori estimates}
In this part we provide some preliminary results. These results can be regarded as mean field games analogues of a priori extimates for Hamilton-Jacobi equations without the coupling term considered in \cite{GT,ST}. For completeness sake, we prove our versions here. The key point is that the estimates are uniform on $m\in\p$.

\begin{lem}\label{es}
	For each $D>0$, there is $C(D)>0$ such that for each $0<\s<1$, each $m\in\p$, each $x$, $y\in\R^d$ with $|x-y|\leqslant \s D$, and each minimizing curve $\g^m_{x,y}:[0,\s]\to\R^d$ of $L_m$ with the action 
	$h_\s^m(x,y)$,
	there hold
	\[
	|\dot{\g}^m_{x,y}(s)|,\ |\ddot{\g}^m_{x,y}(s)|\leqslant C(D),\quad \forall s\in[0,\s].
	\]
\end{lem}

\begin{proof}
	Fix $D>0$. For each $0<\s<1$,  each $x$, $y\in\R^d$ with $|x-y|\leqslant \s D$, let $\ell_{x,y}$ be a segment connecting $x$ and $y$
	\[
	\ell_{x,y}:[0,\s]\to\R^d,\quad \ell_{x,y}(s):=x+s\frac{y-x}{\s}.
	\]
	Then for each $m\in\p$,
\begin{align*}
	\int_0^\s \Big(L(\ell_{x,y}(s),\dot{\ell}_{x,y}(s))+F(\ell_{x,y}(s),m)\Big)ds& =\int_0^\s\Big(L(\ell_{x,y}(s),\frac{y-x}{\s})+F(\ell_{x,y}(s),m)\Big)ds\\
	&\leqslant \s\Big(\max_{x\in\T^d,|v|\leqslant D}L(x,v)+F_\infty\Big)\\
	&=:\s C_1(D).
\end{align*}

Since $L$ is superlinear in $v$, then there is $R>0$ such that
for any $v\in\R^d$ with $|v|>R$,
\[
L(x,v)+F(x,m)>C_1(D), \quad \forall x\in\T^d,\ \forall m\in\p.
\] 	
Let
\[
\Sigma_R:=\{(x,v)\in\T^d\times\R^d: |v|\leqslant R\}.
\]	
Obviously, $\Sigma_R$ is a compact subset of $\T^d\times\R^d$.  By the compactness of $\Sigma_R$ and $\p$, the continuous dependence of the solutions on the initial condition and a parameter and (F1), one can deduce that there is  $R_1>0$ independent of $\s$ and $m$ such that
\[
\Phi^{L_m}_s(\Sigma_R)\subset \Sigma_{R_1}:=\{(x,v)\in\T^d\times\R^d: |v|\leqslant R_1\}
\]
for all $s\in[-1,1]$ and all $m\in\p$.

For any minimizing curve $\g^m_{x,y}:[0,\s]\to\R^d$ of $L_m$ with the action $h_\s^m(x,y)$, we assert that $|\dot{\g}^m_{x,y}(s)|\leqslant R_1$ for all $s\in[0,\s]$. Otherwise, there would be $s_0\in[0,\s]$ such that $|\dot{\g}^m_{x,y}(s_0)|> R_1$. We define a curve $\tilde{\g}$ in $\T^d$ by $\tilde{\g}:=\pi \g^m_{x,y}$. Since  $\g^m_{x,y}$ is a minimizing curve, then we know that $(\tilde{\g}(s),\dot{\tilde{\g}}(s))\subset \T^d\times\R^d$ is a solution of the Lagragian system generated by $L_m$. In view of $|\dot{\tilde{\g}}(s_0)|=|\dot{\g}^m_{x,y}(s_0)|>R_1$, one can deduce that
\[
(\tilde{\g}(s),\dot{\tilde{\g}}(s))\notin\Sigma_{R}\quad \forall s\in[0,\s].
\]
So,
$$|\dot{\tilde{\g}}(s)|=|\dot{\g}^m_{x,y}(s)|>R$$ for all $s\in[0,\s]$. Thus, we have that
\[
L(\g^m_{x,y}(s),\dot{\g}^m_{x,y}(s))+F(\g^m_{x,y}(s),m)>C_1(D), \quad \forall s\in[0,\s]
\]	
implying that
\[
\int_0^\s L(\g^m_{x,y}(s),\dot{\g}^m_{x,y}(s))+F(\g^m_{x,y}(s),m)ds>C_1(D)\s\geqslant \int_0^\s \Big(L(\ell_{x,y}(s),\dot{\ell}_{x,y}(s))+F(\ell_{x,y}(s),m)\Big)ds,
\]	
a contradiction.

At last, note that
\[
\ddot{\g}^m_{x,y}=\frac{\partial^2 L}{\partial v^2}(\g^m_{x,y},\dot{\g}^m_{x,y})^{-1}\Big(\frac{\partial L}{\partial x}(\g^m_{x,y},\dot{\g}^m_{x,y})+\frac{\partial F}{\partial x}(\g^m_{x,y},m)- \frac{\partial^2 L}{\partial x\partial v}(\g^m_{x,y},\dot{\g}^m_{x,y})\Big),
\]
which finishes the proof.
	
\end{proof}

\begin{prop}\label{pr1}
For each $D>0$, there is $\tilde{C}(D)>0$ such that if $\s\in(0,1]$, $x$, $y\in\R^d$ with $|x-y|\leqslant \s D$, then
\[
|h_\s^m(x,y)-\mathcal{L}_{\tau,m}(x,y)|\leqslant \s^2\tilde{C}(D),\quad \forall m\in\p.
\]
\end{prop}

\begin{proof}
Fix $D>0$. Let $C(D)$ be the constant given by Lemma \ref{es}. Let $\s\in(0,1]$ and $x$, $y\in\R^d$ with $|x-y|\leqslant \s D$.
Let $\g^m_{x,y}$ be a minimizing curve $\g^m_{x,y}:[0,\s]\to\R^d$ of $L_m$ with the action $h_\s^m(x,y)$. Then by Lemma \ref{es}, we get that $|\dot{\g}^m_{x,y}(s)|$, $|\ddot{\g}^m_{x,y}(s)|\leqslant C(D)$ for all $s\in[0,\s]$.
For any $s\in[0,\s]$, we have that
\begin{align*}
	& |\g^m_{x,y}(s)-x|=|\g^m_{x,y}(s)-\g^m_{x,y}(0)|\leqslant \s C(D),\quad |\dot{\g}^m_{x,y}(s)-\dot{\g}^m_{x,y}(0)|\leqslant \s C(D),\\[2mm]
	& \Big|\frac{y-x}{\s}-\dot{\g}^m_{x,y}(0)\Big|=\Big|\frac{\g^m_{x,y}(\s)-\g^m_{x,y}(0)-\dot{\g}^m_{x,y}(0)\s}{\s}\Big|\leqslant \s C(D),\\[2mm]
	& |\dot{\g}^m_{x,y}(s)-\frac{y-x}{\s}|\leqslant |\dot{\g}^m_{x,y}(s)-\dot{\g}^m_{x,y}(0)|+|\dot{\g}^m_{x,y}(0)-\frac{y-x}{\s}|\leqslant 2\s C(D).
\end{align*}
So, we get that
\begin{align*}
	|h_\s^m(x,y)-\mathcal{L}_{\tau,m}(x,y)|& \leqslant \int_0^\s\Big|L(\g^m_{x,y}(s),\dot{\g}^m_{x,y}(s))+F(\g^m_{x,y}(s),m)-L(x,\frac{y-x}{\s})-F(x,m)\Big|ds\\
	& \leqslant \int_0^\s\Big|L(\g^m_{x,y}(s),\dot{\g}^m_{x,y}(s))-L(x,\frac{y-x}{\s})\Big|+\Big|F(\g^m_{x,y}(s),m)-F(x,m)\Big|ds\\
	& \leqslant C_2(D)\s^2+F_\infty C(D)\s^2=: \tilde{C}(D)\s^2.
\end{align*}
	
\end{proof}

We use the symbol $A^m_\s(x,y)$ to denote $h^m_\s(x,y)$ or $\mathcal{L}_{\s,m}(x,y)$ in the following four  propositions, which means these results hold for both  $h^m_\s(x,y)$ and $\mathcal{L}_{\s,m}(x,y)$.
 The first one is a direct consequence of assumptions (L1), (L2), (F1) and Lemma \ref{cmp}. We omit the proof here.
\begin{prop}\label{pro}
	$A^m_\s(x,y)$ satisfies the following properties:
	\begin{itemize}
		\item [(i)]  for each $D>0$,
		$$
		\inf_{m\in\p}\inf _{\tau \in(0,1]} \inf_{x, y \in \mathbb{R}^{d}} \frac{1}{\tau} A^m_{\tau}(x, y)>-\infty, \quad \sup_{m\in\p}\sup_{\tau \in(0,1]} \sup_{|y-x|\leqslant \tau D} \frac{1}{\tau} A^m_{\tau}(x, y)<+\infty;
		$$
		\item [(ii)]
		$$
		\lim _{D \rightarrow+\infty} \inf _{\tau \in(0,1]} \inf _{|x-y| \geqslant \tau D} \frac{A^m_{\tau}(x, y)}{|x-y|}=+\infty
		$$
		uniformly on $m\in\p$;
		\item [(iii)]  for each $D>0$, there exists a constant $C(D)>0$ such that for each $\tau \in(0,1]$, for each $x, y, z \in \mathbb{R}^{d}$, and each $m\in\p$,
		\begin{itemize}
			\item [(iii')]	if $|y-x| \leqslant \tau D$ and $|z-x| \leqslant \tau D$, then
			$
				\left|A^m_{\tau}(x, z)-A^m_{\tau}(x, y)\right| \leqslant C(D)|z-y|, $
			\item [(iii'')] if $|z-x| \leqslant \tau D$ and $|z-y| \leqslant \tau D$,  then
			$
			\left|A^m_{\tau}(x,z)-A^m_{\tau}(y,z)\right| \leqslant C(D)|y-x|.
			$
		\end{itemize}
	\end{itemize}
\end{prop}

The following result comes from \cite{GT}, where the authors dealt with Hamilton-Jacobi equations without coupling term $F(x,m)$.
\begin{prop}
	\begin{itemize}
		\item [(i)] For each $\tau>0$ and each $m\in\p$, there exists a unique constant $\bar{A}^m_{\tau}$ such that equation
		\begin{align}\label{de4}
		u_{\s,m}(y)+\bar{A}^m_{\tau}=\inf_{x\in\R^d}\big(u_{\s,m}(x)+A^m_{\s}(x,y)\big), \quad  \forall y\in\R^d,
		\end{align}
		admits a continuous $\mathbb{Z}^d$-periodic solution $u_{\tau,m}.$
		\item [(ii)] $\bar{A}^m_{\tau}$ can be represented by
		\begin{align}\label{for}
			\bar{A}^m_{\tau}=\lim _{k \rightarrow+\infty} \inf _{z_{0}, \ldots, z_{k} \in \mathbb{R}^{d}} \frac{1}{k} \sum_{i=0}^{k-1} A^m_{\tau}\left(z_{i}, z_{i+1}\right).
		\end{align}
	\end{itemize}
\end{prop}

\begin{re}
	Let $\bar{\mathcal{L}}_{\s,m}:=\bar{A}^m_\s$ when $A^m_\s(x,y)=\mathcal{L}_{\s,m}(x,y)$. In view of \eqref{de} and \eqref{de4}, one can deduce that $\frac{\bar{\mathcal{L}}_{\s,m}}{\s}=\bar{L}(\s,m)$.
\end{re}

\begin{prop}\label{pr4}
	There exist constants $C$, $D>0$ such that if $\tau \in(0,1]$ and $m\in\p$ and $u_{\tau,m}$ is a solution of \eqref{de4}, then
	\begin{itemize}
		\item [(i)] $u_{\tau,m}$ is Lipschitz and $\operatorname{Lip}\left(u_{\tau,m}\right) \leqslant C$,
		\item [(ii)] $\forall y \in \mathbb{R}^{d}, x \in \arg \min _{x \in \mathbb{R}^{d}}\left\{u_{\tau}(x)+A^m_{\tau}(x, y)\right\} \Rightarrow |y-x| \leqslant \tau D$.
	\end{itemize}
	\end{prop}

\begin{proof}
	Let
	$$
	\begin{aligned}
		C_{1} &:=2 \sup _{\tau \in(0,1],|y-x| \leqslant \tau,m\in\p} \frac{A^m_{\tau}(x, y)-\bar{A}^m_{\tau}}{\tau}, \\[1mm]
		D &:=\inf \left\{D'>1: \inf _{\tau \in(0,1],|y-x|>\tau D',m\in\p} \frac{A^m_{\tau}(x, y)-\bar{A}^m_{\tau}}{|y-x|}>C_{1}\right\}, \\[1mm]
		C &:=\max \left\{C_{1},\sup_{|y-x|,|z-x| \leqslant \tau(D+1),\s\in(0,1],m\in\p} \frac{A^m_{\tau}(x, y)-A^m_{\tau}(x, z)}{|z-y|}\right\}.
	\end{aligned}
	$$

Notice that the above three constants $C_{1}$, $D$ and $C$ are well defined since $A^m_\s(x,y)$ satisfies (i), (ii), (iii) in Proposition \ref{pro} and $\bar{A}^m_{\tau}$ has the representation formula \eqref{for}.

	First, we show if $|x-y|>\tau$, then  $u_{\tau,m}(y)-u_{\tau,m}(x) \leqslant C_{1}|y-x|$.
	In fact, by choosing $n \geqslant 2$ such that $(n-1) \tau<|y-x| \leqslant n \tau$ and by choosing $x_{i}=x+\frac{i}{n}(y-x)$, we obtain $n \tau \leqslant 2|y-x|$,
\begin{align*}
		u_{\tau,m}\left(x_{i+1}\right)-u_{\tau,m}\left(x_{i}\right) &\leqslant A^m_{\tau}\left(x_{i}, x_{i+1}\right)-\bar{A}^m_{\tau},  \\[2mm]
		u_{\tau,m}(y)-u_{\tau,m}(x) &\leqslant n \tau \sup _{|z-z'| \leqslant \s} \frac{A^m_{\tau}(z, z')-\bar{A}^m_{\tau}}{\tau} \leqslant C_{1}|y-x|.
\end{align*}
Second, we prove (ii). Let $y \in \mathbb{R}^{d}$ and take $x$ satisfying
	$$
	u_{\tau,m}(y)-u_{\tau,m}(x)=A^m_{\tau}(x, y)-\bar{A}^m_{\tau}.
	$$
	Assume by contradiction that $|y-x|>\tau D$. Then the first step of the proof may be used and we obtain that
	$$
	C_{1}|y-x| \geqslant u_{\tau,m}(y)-u_{\tau,m}(x)=A^m_{\tau}(x, y)-\bar{A}^m_{\tau}>C_{1}|y-x|,
	$$
	a contradiction.
	
	Third, we end the proof of (i). Let $y$, $z \in \mathbb{R}^{d}$ with $|z-y| \leqslant \tau$. Let $x$ be a point satisfying $u_{\tau,m}(y)-u_{\tau,m}(x)=A^m_{\tau}(x, y)-\bar{A}^m_{\tau}$. Then $|y-x| \leqslant \tau D$, $|z-x| \leqslant \tau(D+1)$,
	\begin{align*}
		 u_{\tau,m}(z)-u_{\tau,m}(x) & \leqslant A^m_{\tau}(x, z)-\bar{A}^m_{\tau}, \\
		u_{\tau,m}(z)-u_{\tau,m}(y) & \leqslant A^m_{\tau}(x, z)-A^m_{\tau}(x, y) \leqslant C|z-y|.
	\end{align*}
	By changing the roles of  $z$ and $y$, we just have proved that $\operatorname{Lip}\left(u_{\tau,m}\right) \leqslant C$.
		
\end{proof}

\begin{prop}\label{pr5}
	For each constant $\kappa>0$, there exist constants $D_{\kappa}, C_{\kappa}>0$, such that if $\varphi$ is any $\mathbb{Z}^d$-periodic Lipschitz function satisfying $\operatorname{Lip}(\varphi) \leqslant \kappa$, $\tau \in(0,1]$, and $m\in\p$, then
	 \begin{itemize}
	 	\item [(i)] $\forall y \in \mathbb{R}^{d}, x \in \arg \min _{x \in \mathbb{R}^{d}}\left\{\varphi(x)+A^m_{\tau}(x, y)\right\} \Rightarrow|y-x| \leqslant \tau D_{\kappa}$
	 	\item [(ii)] $\left\|\inf_{x\in\R^d}\big(\varphi(x)+A^m_{\s}(x,\cdot)\big)-u(\cdot)\right\|_{\infty} \leqslant \tau C_{\kappa}$.
	 \end{itemize}
\end{prop}	

\begin{proof}
 (i)  Let $\kappa>0$. Define
	$$
	D_{\kappa}:=\inf \left\{D^{\prime}>1: \inf _{\tau \in(0,1],|y-x|>\tau D^{\prime}, m\in\p} \frac{A^m_{\tau}(x, y)-A^m_{\tau}(y, y)}{|y-x|}>\kappa\right\}.
	$$
	Let $\varphi$ be a periodic function satisfying $\operatorname{Lip}(\varphi) \leqslant \kappa$ and $y$ be any point in $\mathbb{R}^{d}$. Let $x$ be a point realizing the minimum of $\min_{x}\left\{\varphi(x)+A^m_{\tau}(x, y)\right\}$. Assume by contradiction that $|y-x|>\tau D_{\kappa}$, then
	$$
	A^m_{\tau}(x, y)-A^m_{\tau}(y, y)>\kappa|y-x|.
	$$
	On the other hand, we have  $\varphi(x)+A^m_{\tau}(x, y) \leqslant \varphi(y)+A^m_{\tau}(y, y)$ and
	$$
	\kappa|y-x| \geqslant \varphi(y)-\varphi(x) \geqslant A^m_{\tau}(x, y)-A^m_{\tau}(y, y),
	$$
	a contradiction.

	(ii) Consider the case $A^m_\s(x,y)=\mathcal{L}_{\s,m}(x,y)$ first.
	For any given $y\in\R^d$, let $x_0\in\R^d$ be a point satisfying $x_0\in \arg \min _{x \in \mathbb{R}^{d}}\left\{\varphi(x)+\mathcal{L}_{\tau,m}(x, y)\right\}$. Then by (i) we get that $|y-x_0|\leqslant \s D_{\kappa}$. Hence,
	\begin{align*}
		\Big|\min _{x \in \mathbb{R}^{d}}\{\varphi(x)+\mathcal{L}_{\tau,m}(x, y)\}-\varphi(y)\Big|&=\Big|\varphi(x_0)-\varphi(y)+\s L(x_0, \frac{y-x_0}{\s})+\s F(x_0,m)\Big|\\
		& \leqslant \kappa|x_0-y|+\s \max_{|v|\leqslant D_\kappa}|L(x,v)|+\s F_\infty\\
		& \leqslant \s (\kappa D_\kappa+ \max_{|v|\leqslant D_\kappa}|L(x,v)|+ F_\infty)\\
		&=:\s C_\kappa.
\end{align*}
Next,  consider the case $A^m_\s(x,y)=h^m_{\s}(x,y)$.
	For any given $y\in\R^d$, let $x_0\in\R^d$ be a point satisfying $x_0\in \arg \min _{x \in \mathbb{R}^{d}}\left\{\varphi(x)+h^m_{\s}(x, y)\right\}$. Then by (i) we get that $|y-x_0|\leqslant \s D_{\kappa}$.
	Let $\g^m_{x_0,y}$ be a minimizing curve of $L_m$ with the action $h^m_{\s}(x_0,y)$.
	Then,
	\begin{align*}
		\Big|\min _{x \in \mathbb{R}^{d}}\{\varphi(x)+h^m_{\s}(x_0, y)\}-\varphi(y)\Big|&=\Big|\varphi(x_0)-\varphi(y)+\int_0^\s L(\g^m_{x_0,y}(s),\dot{\g}^m_{x_0,y}(s))+F(\g^m_{x_0,y}(s),m)ds\Big|\\
		& \leqslant \kappa|x_0-y|+\s \max_{|v|\leqslant C(D_\kappa)}|L(x,v)|+\s F_\infty\\
		& \leqslant \s (\kappa D_\kappa+ \max_{|v|\leqslant C(D_\kappa)}|L(x,v)|+ F_\infty)\\
		&=:\s C_\kappa.
	\end{align*}
	
\end{proof}

\section{Minimizing holonomic measures for mean field games}

\begin{de} For each $\s>0$, each $m\in\p$, the set
	\[
	\mathcal{M}_{\tau}(L_m)=\operatorname{cl}\Big( \bigcup\big\{\supp(\mu): \mu\ \text{is a minimizing}\  \s\text{-holonomic measure for}\ L_{m}\big\} \Big),
	\]
	is called $\s$-Mather set for $L_m$.
\end{de}

A function $\varphi: \mathbb{R}^{d} \rightarrow \mathbb{R}$ is called a $\tau$-sub-action with respect to $L_m$ if $\varphi(x)$ is $\mathbb{Z}^{d}$-periodic, continuous and satisfies
$$
\tau \bar{L}(\tau,m) \leqslant \tau L_m(x, v)+\varphi(x)-\varphi(x+\tau v), \quad \forall(x, v) \in \mathbb{T}^{d} \times \mathbb{R}^{d}.
$$
It is straightforward to check that any solution $u_{\s,m}$ of \eqref{de} is a $\tau$-sub-action with respect to $L_m$.

Define the sets
$$
\mathcal{N}_{\tau}(L_m, u_{\s,m}):=\left\{(x, v) \in \mathbb{T}^{d} \times \mathbb{R}^{d}: \tau L_m(x, v)=u_{\s,m}(x+\tau v)-u_{\s,m}(x)+\tau \bar{L}(\tau,m)\right\}.
$$
By \cite[Proposition 6.3]{GT} we have $\mathcal{M}_{\tau}(L_m) \subset \mathcal{N}_{\tau}(L_m, u_{\s,m})$.

\begin{prop}\label{cmp} There are a compact subset $\mathcal{K}\subset\T^d\times\R^d$ and a constant $\s_0>0$, such that
	$\mathcal{M}_{\tau}(L_m)\subset \mathcal{K}$ for all $0<\s<\s_0$ and all $m\in\p$.
\end{prop}

\begin{proof}
	We show that $\mathcal{N}_{\tau}(L_m, u_{\s,m})$ is a bounded subset of $\T^d\times\R^d$. Note that for any $(x,v)\in\mathcal{N}_{\tau}(L_m, u_{\s,m})$, we have
\begin{align*}
	\tau L_m(x, v) =u_{\s,m}(x+\tau v)-u_{\s,m}(x)+\tau \bar{L}(\tau,m)\leqslant C\s |v|+\tau \bar{L}(\tau,m),
\end{align*}
which implies that
\begin{align}\label{2-1}
	\begin{split}
	L(x,v)+F(x,m) & \leqslant C|v|+\bar{L}(\tau,m)\\
	& \leqslant C|v|+\min _{\mu} \int_{\mathbb{T}^{d} \times \mathbb{R}^{d}} L(x,v) d\mu+F_\infty\\
	& \leqslant C|v|+\bar{L}(\s)+F_\infty,
	\end{split}
\end{align}
where $C$ independent of $\s$ and $m$ is the common Lipschitz constant of $u_{\s,m}$, and the minimum is taken over $\mathcal{P}_{\tau}\left(\mathbb{T}^{d} \times \mathbb{R}^{d}\right)$. Since $\bar{L}(\s)\to -c(H)$ as $\s\to 0$, then there is a constant $R_1>0$ and $\s_0>0$, such that $|\bar{L}(\s)|\leqslant R_1$ for all $\s\in(0,\s_0)$. Recall that $L$ is superlinear in $v$. Then by \eqref{2-1}, there is a constant $R_2>0$ such that
\[
|v|\leqslant R_2.
\]
Hence, we have proved that $\mathcal{N}_{\tau}(L_m, u_{\s,m})\subset \T^d\times B_{R_2}$ for all $\s<\s_0$ and all $m\in\p$.
\end{proof}

\begin{prop}\label{pr7}
	For each $0<\s<\s_0$, there is $m\in\p$ such that there exists a minimizing $\s$-holonomic measure $\mu_{\s,m}$ for the Lagrangian $L_m$ with
	\[
	m=\pi\sharp \mu_{\s,m}.
	\]
	We call such a measure m minimizing $\s$-holonomic  measure for mean field games \eqref{lab1} and denote it by $m_\s$ (maybe not unique).
\end{prop}

 \begin{proof}
 	For each $\s>0$, define a set-valued map as follows:
 	\begin{align*}
 		 \Psi: \ \p&\to\p\\
 		\ m&\mapsto\Psi(m):=\{\pi\sharp\mu: \mu\ \text{is a minimizing}\ \s\text{-holonomic measure for}\ L_m\}.
 	\end{align*}
We will use Kakutani fixed point theorem to get a fixed point of the map $\Psi$. So, we only need to check: (i) $\p$ is convex and compact; (ii) $\Psi$ is upper semicontinuous with nonempty closed convex values.

It is clear that the metric space $(\p,d_1)$ is convex and compact due to Prokhorov theorem. By \cite[Proposition 3.7]{GT}, for each $0<\s<\s_0$ and each $m\in\p$, there exists a minimizing $\s$-holonomic measure for $L_m$ and thus $\Psi(m)$ is nonempty. In view of Proposition \ref{cmp}, it is direct to check that $\Psi(m)$ is closed. The convexity of $\Psi(m)$ follows from the definition of $\Psi$.

Next, we show: if $m_i\stackrel{w^*}{\longrightarrow} m_0$, $\eta_i\in\Psi(m_i)$ and $\eta_i\stackrel{w^*}{\longrightarrow} \eta_0$, then $\eta_0\in\Psi(m_0)$. By definition, there is a sequence of minimizing $\s$-holonomic measures $\{\mu_{m_i}\}$ for $L_{m_i}$ such that
\begin{align}\label{2-6}
\eta_i=\pi\sharp\mu_{m_i}.
\end{align}
From Proposition \ref{cmp}, if necessary passing to a subsequence, we have
\begin{align}\label{2-7}
\mu_{m_i}\stackrel{w^*}{\longrightarrow}\mu_0.
\end{align}
In view of \eqref{2-6}, \eqref{2-7} and $\eta_i\stackrel{w^*}{\longrightarrow} \eta_0$, one can get that
\[
\eta_0=\pi\sharp\mu_0.
\]
So, we only need to show that $\mu_0$ is a minimizing $\s$-holonomic measure for $L_{m_0}$.  By \eqref{5-1},
\[
\bar{L}(\tau,m_i)=\min _{\mu} \int_{\mathbb{T}^{d} \times \mathbb{R}^{d}} L(x,v)+F(x,m_i) d\mu,
\]
and
\[
\bar{L}(\tau,m_0)=\min _{\mu} \int_{\mathbb{T}^{d} \times \mathbb{R}^{d}} L(x,v)+F(x,m_0) d\mu,
\]
where the minimum is taken over $\mathcal{P}_{\tau}\left(\mathbb{T}^{d} \times \mathbb{R}^{d}\right)$. Thus, we get that
\[
|\bar{L}(\tau,m_i)-\bar{L}(\tau,m_0)|\leqslant \mathrm{Lip}(F)d_1(m_i,m_0)\to 0, \quad i\to+\infty.
\]
Since
\[
\bar{L}(\tau,m_i)=\int_{\mathbb{T}^{d} \times \mathbb{R}^{d}} L(x,v)+F(x,m_i) d\mu_{m_i},
\]
and
\begin{align*}
& \Big|\int_{\mathbb{T}^{d} \times \mathbb{R}^{d}} L(x,v)+F(x,m_i) d\mu_{m_i}-\int_{\mathbb{T}^{d} \times \mathbb{R}^{d}} L(x,v)+F(x,m_0) d\mu_{m_0}\Big|\\
 \leqslant &\Big|\int_{\mathbb{T}^{d} \times \mathbb{R}^{d}} L(x,v)+F(x,m_i) d\mu_{m_i}-\int_{\mathbb{T}^{d} \times \mathbb{R}^{d}} L(x,v)+F(x,m_0) d\mu_{m_i}\Big|\\
& +   \Big|\int_{\mathbb{T}^{d} \times \mathbb{R}^{d}} L(x,v)+F(x,m_0) d\mu_{m_i}-\int_{\mathbb{T}^{d} \times \mathbb{R}^{d}} L(x,v)+F(x,m_0) d\mu_{m_0}\Big|\\
\leqslant & \mathrm{Lip}(F)d_1(m_i,m_0)+\Big|\int_{\mathbb{T}^{d} \times \mathbb{R}^{d}} L(x,v)+F(x,m_0) d(\mu_{m_i}-\mu_{m_0})\Big|\to0,
\end{align*}
then we obtain
\[
\bar{L}(\tau,m_0)=\int_{\mathbb{T}^{d} \times \mathbb{R}^{d}} L(x,v)+F(x,m_0) d\mu_{m_0}.
\]
Since $\mu_{m_i}$ are minimizing $\s$-holonomic measures, by  \eqref{2-7} and the compactness of the supports of the measures we deduce that $\mu_0$ is also $\s$-holonomic, and the proof is complete.

 \end{proof}

\section{Convergence to Hamilton-Jacobi equations}

For each $\s>0$, consider solutions of the discrete Lax-Oleinik equation
\begin{align}\label{de1}
	u_{\s,m_\s}(y)+\s\bar{L}(\s,m_\s)=\inf_{x\in\R^d}\big(u_{\s,m_\s}(x)+\mathcal{L}_{\s,m_\s}(x,y)\big), \quad  \forall y\in\R^d.
\end{align}

\begin{prop}\label{hj}
	There is a subsequence $\s_i\to0$, a subsequence $m_{\s_i}\stackrel{w^*}{\longrightarrow} m_0$, and a subsequence $u_{\s_i,m_{\s_i}}$ solutions of \eqref{de1}  such that $u_{\s_i,m_{\s_i}}$ converges to $u_0$ uniformly on $\T^d$. Moreover,  $u_0$ is a viscosity solution of
	\[
	H(x,Du)=F(x,m_0)+c(m_0).
	\]
\end{prop}

\begin{proof}
For each $m\in\p$, define two kinds of one-parameter operators $\mathfrak{T}_{\tau}^m$ and $T^m_{\tau}$ as follows:
\[
\mathfrak{T}_{\tau}^m\varphi(y)=\inf_{x\in\R^d}\big(\varphi(x)+\mathcal{L}_{\s,m}(x,y)\big),
\]
and
\[
T_{\tau}^m\varphi(y)=\inf_{x\in\R^d}\big(\varphi(x)+h^m_{\s}(x,y)\big).
\]
We claim there exists a constant $C>0$ such that for each small $\tau>0$, each $m\in\p$, and  each solution $u$ of equation \eqref{de} with $u(0)=0$,
	$$
	\left\|T^m_{\tau}u-\mathfrak{T}_{\tau}^mu\right\|_{\infty} \leqslant \tau^{2} C.
	$$
	In fact, the above estimate is a consequence of Propositions \ref{pr1}, \ref{pr4} and \ref{pr5}. More precisely, from Propositions \ref{pr1}, \ref{pr4} and \ref{pr5}, there exist positive constants $D$ and $C$ such that for each $\tau \in(0,1]$, each $m\in\p$ and  each solution $u$ of equation \eqref{de}, we have that
	\begin{itemize}
		\item $\mathrm{Lip}(u) \leqslant C,\|u\|_{\infty} \leqslant C$;
		\item $\forall y \in \mathbb{R}^{d}, \quad x \in \arg \min _{x \in \mathbb{R}^{d}}\left\{u(x)+\mathcal{L}_{\tau,m}(x, y)\right\} \Rightarrow|y-x| \leqslant \tau D$;
		\item $\forall y \in \mathbb{R}^{d}, \quad x \in \arg \min _{x \in \mathbb{R}^{d}}\left\{u(x)+h^m_{\tau}(x, y)\right\} \Rightarrow|y-x| \leqslant \tau D$;
		\item $\left\|T^m_{\tau}u-u\right\|_{\infty} \leqslant \tau C$;
		\item for each $x, y\in\R^d$, $|y-x| \leqslant \tau D \Rightarrow | h^m_{\tau}(x, y)-\mathcal{L}_{\tau,m}(x, y) \mid \leqslant \tau^{2} C$.
	\end{itemize}
	
For each $y$ and $x \in \arg \min _{x \in \mathbb{R}^{d}}\left\{u(x)+\mathcal{L}_{\tau,m}(x, y)\right\}$,
	we have $$
	\begin{gathered}
		T_{\tau}^mu(y) \leqslant u(x)+h^m_{\tau}(x, y) \leqslant u(x)+\mathcal{L}_{\tau,m}(x, y)+\tau^{2} C \leqslant \mathfrak{T}^m_{\tau}u(y)+\tau^{2} C.
	\end{gathered}
	$$
On the other hand, if $x \in \arg \min _{x \in \mathbb{R}^{d}}\left\{u(x)+h^m_{\tau}(x, y)\right\}$,
$$
\begin{gathered}
	T^m_{\tau}u(y)=u(x)+h^m_{\tau}(x, y) \geqslant u(x)+\mathcal{L}_{\tau,m}(x, y)-\tau^{2} C\geqslant \mathfrak{T}^m_{\tau}u(y)-\tau^{2} C.
\end{gathered}
$$	
Therefore, the above claim is true.

By the Lipschitz estimate, for each $\tau \in(0,1]$ and each $m\in\p$, one can choose solutions $u_{\tau,m}$ of \eqref{de} such that $u_{\tau,m}(0)=0$ and thus $u_{\tau,m}$ is uniformly bounded in $\tau \in(0,1]$ and $m\in\p$. Thus,
by Ascoli-Arzela theorem and Prokhorov theorem, we can choose a subsequence $m_{\s_i} \stackrel{w^*}{\longrightarrow} m_0\in\p$ and a subsequence $u_{\s_i,m_{\s_i}}\to u_0$ uniformly on $\T^d$.  For brevity, we use $m_i$, $u_i$ to denote $m_{\s_i}$ and $u_{\s_i,m_{\s_i}}$ respectively in the following.

Let $t>0$ be fixed, and $N_{i}$ be integers such that $N_{i} \tau_{i} \leqslant t<\left(N_{i}+1\right) \tau_i$. The non-expansiveness property of the Lax-Oleinik operator $T^m_t$ implies
\begin{align*}
\left\|T^{m_0}_{t}u-T_{N_{i} \tau_{i}}^{m_i}u_{i}\right\|_{\infty}
&\leqslant \left\|T^{m_0}_{t}u-T_{N_{i} \tau_{i}}^{m_0}u_{i}\right\|_{\infty}+\left\|T_{N_{i} \tau_{i}}^{m_0}u_{i}-T_{N_{i} \tau_{i}}^{m_i}u_{i}\right\|_{\infty}.
\end{align*}
Note that by Proposition \ref{pr5}, we get
\[
\left\|T^{m_0}_{t}u-T_{N_{i} \tau_{i}}^{m_0}u_{i}\right\|_{\infty}\leqslant \left\|T^{m_0}_{t-N_{i} \tau_{i}}u-u\right\|_{\infty}+\left\|u-u_{i}\right\|_{\infty}\to0,\quad i\to+\infty.
\]
Notice that
\begin{align*}
\left\|T_{N_{i} \tau_{i}}^{m_0}u_{i}-T_{N_{i} \tau_{i}}^{m_i}u_{i}\right\|_{\infty}&=
\left\|\inf_{x\in\R^d}\big(u_i(x)+h^{m_0}_{N_{i} \tau_{i}}(x,\cdot)\big)-\inf_{x\in\R^d}\big(u_i(x)+h^{m_i}_{N_{i} \tau_{i}}(x,\cdot)\big)\right\|_{\infty}\\
&\leqslant N_i\s_i \mathrm{Lip}(F)d_1(m_i,m_0)\to0,\quad i\to+\infty.
\end{align*}
So, we have proved that
\[
\left\|T^{m_0}_{t}u-T_{N_{i} \tau_{i}}^{m_i}u_{i}\right\|_{\infty}\to0,\quad i\to+\infty.
\]

Consider
\[
\frac{\bar{\mathcal{L}}_{\s_i,m_i}}{\s_{i}}+c(m_0)=\frac{\bar{\mathcal{L}}_{\s_i,m_i}}{\s_{i}}-\frac{\bar{\mathcal{L}}_{\s_i,m_0}}{\s_{i}}+\frac{\bar{\mathcal{L}}_{\s_i,m_0}}{\s_{i}}+c(m_0).
\]
Note that
\begin{align*}
\Big|\frac{\bar{\mathcal{L}}_{\s_i,m_i}}{\s_{i}}-\frac{\bar{\mathcal{L}}_{\s_i,m_0}}{\s_{i}}\Big|&=|\bar{L}(\tau_i,m_i)-\bar{L}(\tau_i,m_0)|\\
& =\Big|\min _{\mu} \int_{\mathbb{T}^{d} \times \mathbb{R}^{d}} L(x,v)+F(x,m_i) d\mu-\min _{\mu} \int_{\mathbb{T}^{d} \times \mathbb{R}^{d}} L(x,v)+F(x,m_0) d\mu\Big|\\
&\leqslant \mathrm{Lip}(F)d_1(m_i,m_0)\to0,\quad i\to+\infty,
\end{align*}
and by \cite[Theorem 9]{ST},
\[
\Big|\frac{\bar{\mathcal{L}}_{\s_i,m_0}}{\s_{i}}+c(m_0)\Big|\leqslant \s_i C.
\]
Thus, we obtain that
\begin{align}\label{2-10}
\Big|\frac{\bar{\mathcal{L}}_{\s_i,m_i}}{\s_{i}}+c(m_0)\Big|\leqslant \mathrm{Lip}(F)d_1(m_i,m_0)+\s_i C.
\end{align}
The previous claim $\left\|T^{m_i}_{\tau_{i}}u_{i}-\mathfrak{T}^{m_i}_{\tau_{i}}u_{i}\right\|_{\infty} \leqslant \tau_{i}^{2} C$ and \eqref{2-10} imply that
\[
\left\|T^{m_i}_{\tau_{i}}u_{i}-u_{i}+\s_ic(m_0)\right\|_{\infty}\leqslant\left\|T^{m_i}_{\tau_{i}}u_{i}-u_i-\bar{\mathcal{L}}_{\s_i,m_i}\right\|_{\infty}+\Big|\bar{\mathcal{L}}_{\s_i,m_i}+\s_ic(m_0)\Big|\leqslant \s_i\mathrm{Lip}(F)d_1(m_i,m_0)+2\s_i^2 C.
\]
By iterating this inequality, we have
$$
\left\|T^{m_i}_{N_{i} \tau_{i}}u_{i}-u_{i}+N_{i} \tau_{i} c(m_0)\right\|_{\infty} \leqslant N_{i} (\s_i\mathrm{Lip}(F)d_1(m_i,m_0)+2\s_i^2 C) \leqslant t(\mathrm{Lip}(F)d_1(m_i,m_0)+2\s_i C).
$$
Since $u_{i}+N_{i} \tau_{i} c(m_0) \rightarrow u_0+t c(m_0)$, we get
$$
T_{t}^{m_0}u_0=u_0-tc(m_0), \quad \forall t>0.
$$

\end{proof}

\section{Convergence to continuity equations}

\begin{prop}\label{m}
	There is a measure $\mu_0\in\mathcal{P}\left(\mathbb{T}^{d} \times \mathbb{R}^{d}\right)$ such that $\mu_{\s_i,m_{\s_i}}\stackrel{w^*}{\longrightarrow}  \mu_0$ with $m_0=\pi\sharp\mu_0$, where $\s_i$, $m_{\s_i}$ and $m_0$ are as in Proposition \ref{hj}.
	\end{prop}

\begin{proof}
	By Proposition \ref{cmp}, the sequence of $\mu_{\s_i,m_{\s_i}}$ is tight. In view of Prokhorov theorem, there is $\mu_0\in\mathcal{P}\left(\mathbb{T}^{d} \times \mathbb{R}^{d}\right)$ such that $\mu_{\s_i,m_{\s_i}}\stackrel{w^*}{\longrightarrow}  \mu_0$. Note that 	
	\[
	m_{\s_i}=\pi\sharp \mu_{\s_i,m_{\s_i}},\quad m_{\s_i}\stackrel{w^*}{\longrightarrow} m_0,\quad i\to+\infty.
	\]
	One can deduce that $m_0=\pi\sharp\mu_0$.
\end{proof}

\begin{prop}
	$\mu_0$ is a Mather measure for $L_{m_0}$.
\end{prop}


\begin{proof}
	First we prove that
	\begin{align}\label{2-3}
	\int_{\T^d\times\R^d}L_{m_0}(x,v)d\mu_0=-c(m_0).
	\end{align}
Recall that
\[
\bar{L}(\s_i,m_{\s_i})=\int_{\T^d\times\R^d}L(x,v)+F(x,m_{\s_i})d\mu_{\s_i,m_{\s_i}}.
\]	
Note that
\begin{align*}
	& \Big|\int_{\T^d\times\R^d}L(x,v)+F(x,m_{\s_i})d\mu_{\s_i,m_{\s_i}}-\int_{\T^d\times\R^d}L(x,v)+F(x,m_0)d\mu_0\Big|\\
 \leqslant 	& \Big|\int_{\T^d\times\R^d}L(x,v)d\mu_{\s_i,m_{\s_i}}-\int_{\T^d\times\R^d}L(x,v)d\mu_0\Big|+\Big|\int_{\T^d\times\R^d}F(x,m_{\s_i})d\mu_{\s_i,m_{\s_i}}-\int_{\T^d\times\R^d}F(x,m_0)d\mu_{\s_i,m_{\s_i}}\Big|\\
	& + \Big|\int_{\T^d\times\R^d}F(x,m_0)d\mu_{\s_i,m_{\s_i}}-\int_{\T^d\times\R^d}F(x,m_0)d\mu_0\Big|.
\end{align*}
Since $\mu_{\s_i,m_{\s_i}}\stackrel{w^*}{\longrightarrow}  \mu_0$ and $m_{\s_i}\stackrel{w^*}{\longrightarrow} m_0$, then the first and the third terms in the right hand side of the above inequality go to 0 as $i\to+\infty$. We take care of the second term as follows:
\[
\Big|\int_{\T^d\times\R^d}F(x,m_{\s_i})d\mu_{\s_i,m_{\s_i}}-\int_{\T^d\times\R^d}F(x,m_0)d\mu_{\s_i,m_{\s_i}}\Big|\leqslant\mathrm{Lip}(F)d_1(m_{\s_i},m_0)\to0,
\]
as $i\to+\infty$. So, we get that
\[
\bar{L}(\s_i,m_{\s_i})\to \int_{\T^d\times\R^d}L(x,v)+F(x,m_0)d\mu_0.
\]
To finish the proof of \eqref{2-3}, it suffices to show that 	
\begin{align}\label{2-2}
\bar{L}(\s_i,m_{\s_i})\to -c(m_0).
\end{align}
Note that
\begin{align*}
	&|\bar{L}(\s_i,m_{\s_i})+c(m_0)|\\
	  \leqslant& |\bar{L}(\s_i,m_{\s_i})-\bar{L}(\s_i,m_0)|+|\bar{L}(\s_i,m_0)-c(m_0)|\\
	 \leqslant& \sup_{\mu}\Big|\int_{\T^d\times\R^d}L(x,v)+F(x,m_{\s_i})d\mu-\int_{\T^d\times\R^d}L(x,v)+F(x,m_0)d\mu\Big|+ |\bar{L}(\s_i,m_0)-c(m_0)|\\
\leqslant & \sup_{\mu}\int_{\T^d\times\R^d}|F(x,m_{\s_i})-F(x,m_0)|d\mu+ |\bar{L}(\s_i,m_0)-c(m_0)|\\
\leqslant & \mathrm{Lip}(F)d_1(m_{\s_i},m_0)+ |\bar{L}(\s_i,m_0)-c(m_0)|,
\end{align*}
where the supremum is taken over $\mathcal{P}_{\tau_i}\left(\mathbb{T}^{d} \times \mathbb{R}^{d}\right)$. Letting $i\to+\infty$, we get \eqref{2-2}.

Next, we only need to show that $\mu_0$ is a closed measure, i.e.,
\begin{align}\label{2-4}
\int_{\T^d\times\R^d}vD\varphi(x)d\mu_0=0,\quad \forall \varphi\in C^1(\T^d).
\end{align}
Since $C^2(\T^d)$ is a dense subset of $C^1(\T^d)$, it suffices to show \eqref{2-4} holds for each $\varphi\in C^2(\T^d)$.

For each $\s$ and each $\varphi\in C^2(\T^d)$, define
\[
\Delta\varphi_\s(x,v):=\frac{\varphi(x+\s v)-\varphi(x)}{\s}, \quad \forall (x,v)\in\T^d\times\R^d.
\]
It is clear that
\[
\lim_{\s\to0}\Delta\varphi_\s(x,v)=vD\varphi(x),\quad \forall (x,v)\in\T^d\times\R^d.
\]
In fact, for any compact subset $\mathcal{K}'$ of $\T^d\times\R^d$, one can deduce that
\begin{align}\label{2-5}
\lim_{\s\to0}\Delta\varphi_\s(x,v)=vD\varphi(x)
\end{align}
uniformly on $\mathcal{K}'$.

Note that
\begin{align*}
& \Big|\int_{\T^d\times\R^d}\Delta\varphi_{\s_i}(x,v)d\mu_{\s_i,m_{\s_i}}-	\int_{\T^d\times\R^d}vD\varphi(x)d\mu_0\Big|\\
 \leqslant & \Big|\int_{\T^d\times\R^d}\Delta\varphi_{\s_i}(x,v)d\mu_{\s_i,m_{\s_i}}-	\int_{\T^d\times\R^d}vD\varphi(x)d\mu_{\s_i,m_{\s_i}}\Big|+\Big|\int_{\T^d\times\R^d}vD\varphi(x)d\mu_{\s_i,m_{\s_i}}-	\int_{\T^d\times\R^d}vD\varphi(x)d\mu_0\Big|.
\end{align*}
Recall that
$\mu_{\s_i,m_{\s_i}}\stackrel{w^*}{\longrightarrow}  \mu_0$ as $i\to+\infty$ and Corollary \ref{cmp} and \eqref{2-5}. We get that
\[
\lim_{i\to+\infty}\int_{\T^d\times\R^d}\Delta\varphi_{\s_i}(x,v)d\mu_{\s_i,m_{\s_i}}=	\int_{\T^d\times\R^d}vD\varphi(x)d\mu_0.
\]
Since $\mu_{\s_i,m_{\s_i}}$ is a minimizing $\s_i$-holonomic  measure, then
\[
\frac{1}{\s_i}\int_{\T^d\times\R^d}\varphi(x+\tau_i v) - \varphi(x) d\mu_{\s_i,m_{\s_i}}=0.
\]
So, we have that
\[
0=\lim_{i\to+\infty}\int_{\T^d\times\R^d}\Delta\varphi_{\s_i}(x,v)d\mu_{\s_i,m_{\s_i}}=	\int_{\T^d\times\R^d}vD\varphi(x)d\mu_0.
\]
The proof is complete.

\end{proof}

The last result of this paper is well known, see for example \cite{HW}.
\begin{prop}
	$m_0$ is a solution of $\text{div}\Big(m \frac{\partial H}{\partial p}(x, Du_0)\Big)=0$ in the sense of distributions.
\end{prop}

\begin{proof}
Let $\Phi_{t}^{H_{m_0}}$ denote the Hamiltonian flow of $H_{m_0}$.
For any $x\in \supp(m_0)$, let $\gamma_{t}(x)=\pi\circ \Phi_{t}^{H_{m_0}}(x,Du_0(x))$.
	 Then, we have that
	 $$\frac{d}{dt}\gamma_{t}(x)=\frac{\partial H_{m_0}}{\partial p}\left(\gamma_{t}(x), Du_0(\gamma_{t}(x))\right).$$
	 Since the map $\pi: \supp(\mu_0) \to \supp(m_0)$ is one-to-one and its inverse is given by $x \mapsto (x, Du_0(x)))$ on $\supp(m_0)$, then $\gamma_t:\supp(m_0)\to \supp(m_0)$ is a bijection for each $t\in\R$. Note that, for each $t\in\R$ and any function $f \in C^{1}(\T^d)$, we get that
	 \begin{align*}
	 	 	\begin{split}
	 \int_{\supp(m_0)}f(\gamma_t(x))dm_0&=\int_{\supp(m_0)}f\circ \gamma_t(x)d\pi\sharp\mu_0\\&=\int_{\supp(\mu_0)}f\circ \gamma_t(\pi(x,p))d\mu_0\\
	 &=\int_{\supp(\mu_0)}f (\pi\circ \Phi^{H_{m_0}}_t(x,p))d\mu_0\\
	 &=\int_{\supp(\mu_0)}f (\pi(x,p))d\mu_0\\
	 &=\int_{\supp(\mu_0)}f (x)dm_0.
	 \end{split}
	 \end{align*}
	 Here, the first equality holds  since $m_0=\pi\sharp \mu_0$, the second one holds by the property of the push-forward, the third holds since $\gamma_t$ is a bijection, the fourth one comes from the $\Phi^{H_{m_0}}_t$-invariance property of $\mu_0$, and the last one is again due to the property of the push-forward.
	 So, for any function $f \in C^{1}(\T^d)$  and any $t\in\R$, one can deduce that
	 \begin{align*}
	 0 &=\frac{d}{dt}\int_{\T^d}{f(\gamma_{t}(x))\ dm_0(x)}= \int_{\T^d}{\big\langle Df(\gamma_{t}(x)), \frac{\partial H_{m_0}}{\partial p}(\gamma_{t}(x),Du_0(\gamma_{t}(x))) \big\rangle\ dm_0(x)} \\&= \int_{\T^d}{\big\langle Df(x), \frac{\partial H_{m_0}}{\partial p}(x,Du_0(x)) \big\rangle\ d m_0(x)}.
	 \end{align*}
	 Hence, $m_0$ satisfies the continuity equation which  completes the proof.
\end{proof}

\medskip
\noindent {\bf Acknowledgements:}

Renato Iturriaga was partly supported by Conacyt Mexico (Grant No. A1-S-33854).
Kaizhi Wang was partly supported by National Natural Science Foundation of China (Grant No. 12171315).



\end{document}